\documentclass[12pt, reqno]{amsart}

\usepackage{amsfonts,amssymb,latexsym,amsmath, amsxtra}
\usepackage{verbatim}

\usepackage{amsfonts,amssymb,latexsym,amsmath, amsxtra,mathtools}
\usepackage[all]{xy}
\usepackage{graphicx}
\usepackage{color}
\newcommand\N{\mathbb{N}}
\newcommand\C{\mathbb{C}}
\newcommand\Z{\mathbb{Z}}
\newcommand\Ha{\mathbb{H}}
\newcommand{\of}[1]{\left(#1\right)}
\renewcommand\theta{\vartheta}
\newcommand\smod[1]{\ \left(\operatorname{mod} #1\right)}
\newcommand\res{\operatornamewithlimits{Res}}
\newcommand\ko{\hspace{.1em}}
\newtheorem{theorem}{Theorem}[section]
\newtheorem{corollary}[theorem]{Corollary}
\newtheorem{lemma}[theorem]{Lemma}

\theoremstyle{definition}
\newtheorem{definition}[theorem]{Definition}
\newtheorem*{remark}{Remark} 
\numberwithin{equation}{section}
\title[negative index meromorphic Jacobi forms]{On the Fourier coefficients of negative index meromorphic Jacobi forms}
\author{Kathrin Bringmann, Larry Rolen, and Sander Zwegers}
\address{Mathematical Institute\\University of
Cologne\\ Weyertal 86-90 \\ 50931 Cologne \\Germany}
\email{kbringma@math.uni-koeln.de} 
\email{sander.zwegers@uni-koeln.de}
\address{212 McAllister Building \\
The Pennsylvania State University \\ 
University Park, PA 16802}
\email{larryrolen@psu.edu}

\date{\today}
\thanks{The research of the first author was supported by the Alfried Krupp Prize for Young University Teachers of the Krupp foundation and the research leading to these results has received funding from the European Research Council under the European Union's Seventh Framework Programme (FP/2007-2013) / ERC Grant agreement n. 335220 - AQSER. The second author thanks the University of Cologne and the DFG for their generous support via the University of Cologne postdoc grant DFG Grant D-72133-G-403-151001011. }
\begin{document}
\maketitle
\begin{abstract}
In this paper, we consider the Fourier coefficients of meromorphic Jacobi forms of negative index. This extends recent work of Creutzig and the first two authors for the special case of Kac--Wakimoto characters which occur naturally in Lie theory, and yields, as easy corollaries, many important PDEs arising in combinatorics such as the famous rank--crank PDE of Atkin and Garvan. Moreover, we discuss the relation of our results to partial theta functions and quantum modular forms as introducted by Zagier, which together with previous work on positive index meromorphic Jacobi forms illuminates the general structure of the Fourier coefficients of meromorphic Jacobi forms. 

\end{abstract}
\section{Introduction and statement of results}\label{IntroS}
The general framework of Jacobi forms was laid down by Eichler and Zagier in \cite{EichlerZagier}. This theory has played an important role in many areas of number theory, including the theory of Siegel modular forms \cite{Pyatetskii-Shapiro}, the study of central $L$--values and derivatives of twisted elliptic curves \cite{GKZ}, and in the theory of umbral moonshine \cite{ChengHarveyDuncan}, just to name a few. Roughly speaking, a {\it Jacobi form} is a function $\phi\colon\C\times\Ha\rightarrow\C$, where $\Ha:=\{\tau\in\C\colon \mathrm{Im}(\tau)>0\}$, which satisfies two transformations similar to the transformations of elliptic functions and of modular forms (see Section \ref{JacobiS}). We refer to the variable in $\C$ (denoted by $z$) as the \emph{elliptic variable}, and to the variable in $\Ha$ (denoted by $\tau$) as the \emph{modular variable}. As any Jacobi form $\phi$ is one-periodic as a function of $z$, it is natural to consider its Fourier expansion in terms of $\zeta:=e(z)$, where $e(x):=e^{2\pi ix}$. In the classical case of holomorphic Jacobi forms, the Fourier coefficients give rise to a vector-valued modular form via the {\it theta decomposition} of the Jacobi form (see Section \ref{JacobiS}).

If $\phi$ has poles in the elliptic variable, the story becomes much more interesting and difficult. In this case, the Fourier coefficients depend on the choice of range of $z$ and are not modular. Such coefficients played a key role in the study of the mock theta functions of Ramanujan in \cite{zwegers}, where they were studied in relation to mock modular forms and certain Appell--Lerch sums. Subsequent extensions and applications to quantum black holes were given in \cite{DMZ} (see also \cite{ManschotMoore} for the appearance of mock modular forms in the
context of quantum gravity partition functions and AdS3/CFT2, as well
as \cite{Manschot} for a relation between multi-centered black holes and mock
Siegel--Narain theta functions). 
Meromorphic Jacobi forms also played a key role in the study of  
Kac and Wakimoto characters (see \cite{KacWakimoto}), as studied in \cite{BringmannFolsomKacWaki,Folsom,Ol}. Collectively, these works completed the picture in the case when the meromorphic Jacobi form has positive index.

In \cite{bringmann}, Creutzig and the first two authors considered the first cases of negative index Jacobi forms, finishing the question of Kac and Wakimoto for their characters. In particular, using the classical Jacobi theta function $\vartheta(z;\tau)$ as defined in (\ref{Jacobit}), the $(M,N)$--th Kac--Wakimoto character is, for $M,N\in\N_0$ and after a change of variables, the function 
\[\phi_{M,N}(z)=\phi_{M,N}(z;\tau):=\frac{\vartheta\left(z+\frac12;\tau\right)^M}{\vartheta(z;\tau)^N}.\]
As an example of their importance, these functions contain information about certain affine vertex algebras and their associated affine Lie algebras  as
 studied by Kac and Wakimoto \cite{KacWakimoto}, who asked for the general modularity properties of such functions.

Furthermore, for various choices of $M,N$, the functions $\phi_{M,N}$ are of combinatorial interest. In particular, the function $\phi_{0,1}$ is essentially the famous Andrews--Dyson--Garvan crank generating function, which was used by Andrews and Garvan \cite{AndrewsGarvanCrank} to provide a combinatorial explanation for the Ramanujan congruences for the partition function, as postulated by Dyson \cite{Dyson}. Hence, an explicit understanding of the Fourier coefficients of $\phi_{0,N}$ gives relations between powers of the crank generating function and certain Appell--Lerch series, giving a family of PDEs generalizing the ``rank-crank PDE'' of Atkin and Garvan \cite{AG} (see Corollary \ref{RankCrankCor}), and generalizing families of PDEs studied by Chan, Dixit, and Garvan in \cite{ChanDixitGarvan} and by the third author in \cite{ZwegersRankCrankPDE}. The beautiful identity of Atkin and Garvan gives a surprising connection between the rank and crank generating functions which can be used to show various congruences relating ranks and cranks, as well as useful relations between the rank and crank moments \cite{AG}. 

Further examples of negative index Jacobi forms may also been found in the theory of vertex operator algebras. For example, they arise in the context of certain chiral $2$-point functions associated to lattice theories whose trace is restricted to a simple module of a Heisenberg vertex operator algebra. The interested reader is referred to \cite{MasonTuiteTorus} Corollary 3.15 for details, and more details can also be found in \cite{MasonTuiteFree} and \cite{MasonTuiteZuevsky}.

In this paper, we generalize the work of \cite{bringmann}, offering a completely general picture for negative index Jacobi forms. To describe our results, we let $m\in-\frac12\N$, $\tau\in\Ha$, and $\varepsilon\in\{0,1\}$, and consider meromorphic functions $\phi: \C\to \C$ that satisfy the elliptic transformation law \eqref{elliptictrans}.
For example, if $\phi_{M,N}$ is a Kac--Wakimoto character,
then it transforms according to \eqref{elliptictrans} with $\varepsilon=\varepsilon(N)$ and $m=\frac{M-N}2$, where $\varepsilon(N)\in\{0,1\}$ is such that $\varepsilon(N)\equiv N\pmod2$. Note that a Jacobi form also satisfies a modular transformation law (in the suppressed variable $\tau$), but for our main result, only assuming \eqref{elliptictrans} suffices.

We now define $\mathcal D_x:=\frac1{2\pi i}\frac{\partial}{\partial x}$ for a general variable $x$, and consider the level $2M$ Appell--Lerch sum given for $M\in\frac12\N$ by 

\begin{equation}\label{AppellLerchDefn} F_{M,\varepsilon}(z,u)=F_{M,\varepsilon}(z,u;\tau):= \left(\zeta w^{-1}\right)^{M}\sum_{n\in\Z}\frac{(-1)^{n\varepsilon}w^{-2Mn}q^{Mn(n+1)}}{1-q^n\zeta w^{-1}},\end{equation}
where $q:=e(\tau),\, w:=e(u)$. 
The following is then our main result, where, as is further explained in Section \ref{JacobiS}, $D_{j,v}=D_{j,v}(\tau)$ is the $-j$--th Laurent coefficient of $\phi$ around $z=v$, $s_{z_0,\tau}$ gives the locations of a set of representatives of the poles of $\phi$, and $P_{z_0}$ is a fundamental parallelogram for the lattice $\Z\tau+\Z$. Further note that in the following theorem, although the dependence on $\tau$ is suppressed, both sides of \eqref{mainthm1eqn} depend on $\tau$.

\begin{theorem}\label{mainthm1}
Let $m\in-\frac12\N$ and $\varepsilon\in\{0,1\}$, and suppose that $z_0$ is chosen so that $\phi$ has no poles on $\partial P_{z_0}$. If $\phi$ is a meromorphic function satisfying \eqref{elliptictrans} with this particular choice of $\varepsilon$, then 
\begin{equation}\label{mainthm1eqn}\phi(z)=- \sum_{u\in s_{z_0,\tau}}\sum_{n\in\N}\frac{D_{n,u}}{(n-1)!}\mathcal D_v^{n-1}\left(F_{-m,\varepsilon}(z,v)\right)\big\vert_{v=u}.\end{equation}
\end{theorem}

\begin{remark}
As $\phi$ is a meromorphic function, there are only finitely many non-zero terms in the sum over $n$ in the right hand side of \eqref{mainthm1eqn}.
\end{remark}
\begin{remark}
This theorem complements work in \cite{BringmannRaumRichter}. Namely, the authors of \cite{BringmannRaumRichter} show that general H-harmonic Maass--Jacobi forms of index $m<0$ may be decomposed as 
\[\sum_{\ell\pmod{2m}}h_{\ell}(\tau)\widehat\mu_{m,\ell}(z;\tau)+\psi(z;\tau),\]
where the $h_{\ell}$ are the components of a vector-valued modular form, the $\widehat{\mu}_{m,\ell}$ are distinguished functions depending only on $m$ and $\ell$, and $\psi$ is a meromorphic Jacobi form of index $m$ 
(the interested reader is also referred to \cite{Raum} for important extensions of this work). Hence, Theorem \ref{mainthm1} allows one to  further decompose the pieces $\psi$ in these decompositions, which yields explicit decompositions of H-harmonic Maass--Jacobi forms. 
\end{remark}

As a corollary, applying this result to the Kac--Wakimoto characters $\phi_{M,N}$ yields the following, which extends Theorem 1.3 in \cite{bringmann} to the case of general Kac--Wakimoto characters. Note that the only pole of these functions occurs at
$ z = 0 $ (independent of $\tau$), and is of order precisely $N$.
\begin{corollary}\label{OldThm13Cor}
For $M\in\N_0$ and $N\in\N$ with $M<N$, we have the decomposition
\[\phi_{M,N}(z)=- \sum_{n=1}^{N}\frac{D_{n,0}}{(n-1)!}\mathcal D_v^{n-1}\left(F_{\frac {N-M}2,\, \varepsilon(N)}(z,v)\right)\bigg\vert_{v=0}.\]
\end{corollary}

As discussed in \cite{bringmann}, Corollary \ref{OldThm13Cor} has applications to interesting differential equations of combinatorial generating functions, and in particular recovers important identities which were previously observed.
In particular, Theorem \ref{mainthm1} immediately implies the rank crank PDE of Atkin and Garvan. To state it, we first recall the rank and crank generating functions (whose combinatorial definitions are not needed in this paper), which arise in many contexts and in particular give combinatorial explanations of Ramanujan's congruences (for example see \cite{AndrewsGarvanCrank,AtkinSDRank,Dyson}). Specifically, the generating functions $\mathcal{R}$ and $\mathcal{C}$ for the rank and crank, respectively, may be shown to possess the following representations \cite{AndrewsGarvanCrank,AtkinSDRank}, where for $n\in\N_0\cup\{\infty\}$, we set $(a;q)_n=(a)_n:=\prod_{j=0}^{n-1}\left(1-aq^j\right)$: 

\[\mathcal R(\zeta;q):=\sum_{n\geq0}\frac{q^{n^2}}{\of{\zeta q}_n\of{\zeta^{-1}q}_n},\text{  and   }\  \mathcal C(\zeta;q):=\frac{(q)_{\infty}}{(\zeta q)_{\infty}(\zeta^{-1}q)_{\infty}}. \]
Atkin and Garvan \cite{AG} proved the following PDE, where $\mathcal R^*$ and $\mathcal C^*$ are normalized versions of $\mathcal R$ and $\mathcal C$ (this PDE also follows as special cases of the main results in \cite{ChanDixitGarvan,ZwegersRankCrankPDE}). As explained in \cite{bringmann}, the following follows directly from Theorem \ref{mainthm1} when applied to the Jacobi form $\phi_{0,3}$.
\begin{corollary}\label{RankCrankCor} The rank-crank PDE of Atkin and Garvan holds true. That is, we have

\begin{equation}\label{rankcrankPDE} 2\eta(\tau)^2\mathcal C^*(\zeta;q)^3=\left(6\mathcal D_\tau+\mathcal D_z^2\right)\mathcal R^*(\zeta;q).\end{equation}
\end{corollary}

\begin{remark}
Theorem \ref{mainthm1} immediately implies other PDEs for combinatorial generating functions. For example, the results in Section 3.2 of \cite{BringmannZwegers} in relation to the overpartition generating function immediately follow from Theorem \ref{mainthm1} as applied to $\phi_{1,3}$.
\end{remark}

As stated above, our main goal is to describe the Fourier coefficients of $\phi$. We first require partial theta functions defined for $z\in\C$, $\tau\in\mathbb H$, $M\in\frac12\N$, and $\ell\in M+\Z$ by

\begin{equation}\label{PartialThetaDefn} \theta^+_{\ell,\varepsilon,M}(z)=\theta^+_{\ell,\varepsilon,M}(z;\tau)  := \sum_{n\geq0}(-1)^{n\varepsilon}q^{\frac{\of{2Mn-\ell}^2}{4M}}\zeta^{2Mn-\ell}.\end{equation}
Our second result is then the following, where 
$h_{\ell,z_0}(\tau)$ is the $\ell$-th Fourier coefficient of $\phi$ with respect to $z_0$ as defined in \eqref{FourierCoeffDefn}.

\begin{theorem}\label{mainthm2}
Let $m\in-\frac12\N$, $\tau\in\Ha$, and $\phi$ be a meromorphic function satisfying \eqref{elliptictrans} with $\varepsilon\in\{0,1\}$. If $z_0\in\C$ is chosen so that $\phi$ has no poles on $\partial P_{z_0}$,  then we have for any $\ell\in m+\Z$ that
\begin{equation}\label{mainthm2eqn}
h_{\ell,z_0}(\tau) =  \sum_{u\in s_{z_0,\tau}}\sum_{n\in\N}\frac{D_{n,u}(\tau)}{(n-1)!}\mathcal D_z^{n-1}\left(\theta^+_{\ell,\varepsilon,-m} \left(z;\tau\right)\right)\big\vert_{z=u}.
\end{equation}
\end{theorem}

In particular, Theorem \ref{mainthm2} directly implies the following result, which is analogous to Theorem 1.4 of \cite{bringmann} (where a different range for the Fourier coefficients is used). 
\begin{corollary}\label{OldThm14Cor}
Let $\phi=\phi_{M,N}$ with $M\in\N_0$, $N\in\N$, and $M<N$. Then, for any $\ell\in\frac{M-N}2+\Z$, we have 
\[h_{\ell,-\frac12-\frac{\tau}2}(\tau)= \sum_{n=1}^{N}\frac{D_{n,0}(\tau)}{(n-1)!}\mathcal D_z^{n-1}\left(\theta^+_{\ell,\varepsilon(N),\frac{N-M}2} (z;\tau)\right)\Big\vert_{z=0}.\]
\end{corollary}

\begin{remark}
Following the proof of Theorem 1.5 of \cite{bringmann}, one finds that the partial theta functions $\vartheta^+_{\ell,\varepsilon-m}$ are all quantum modular forms, so that Theorem \ref{mainthm2} implies that the Fourier coefficients of a general negative index Jacobi form are expressible as derivatives of quantum modular forms times quasimodular forms. This is further explained in Section \ref{QuantumS} (see Theorem \ref{QuantumPartial}).
\end{remark}

The paper is organized as follows. In Section \ref{JacobiS}, we review the basic theory of Jacobi forms, theta decompositions, and the definition of Fourier coefficients of Jacobi forms. In Section \ref{QuantumS}, we discuss the theory of quantum modular forms in the context of partial theta functions. We complete the proofs of the main results in Section \ref{ProofS}. 

\section*{Acknowledgement}
The authors thank the referee for many comments which improved the exposition of this paper.

\section{Preliminaries}\label{PrelimS}

\subsection{Jacobi forms and Fourier coefficients}\label{JacobiS}
We begin by recalling the notion of a holomorphic Jacobi form, referring the interested reader to \cite{EichlerZagier} for the general theory. Roughly speaking, a \emph{holomorphic Jacobi form} is a holomorphic function $\phi$ on $\C\times\Ha$ which satisfies a \emph{modular transformation}, together with an \emph{elliptic transformation}. We are mainly interested in the explicit elliptic transformation property, and  consider functions which transform in the complex $z$-variable as

\begin{equation}\label{elliptictrans}
\phi(z+\lambda\tau+\mu) = (-1)^{2m\mu+\lambda\varepsilon}e^{-2\pi im(\lambda^2\tau+2\lambda z)}\ko \phi(z)
\end{equation}
for all $\lambda,\mu\in\Z$, where $m\in\frac12\Z$ and $\varepsilon\in\{0,1\}$. We refer to $m$ as the index of the Jacobi form.
For example, in the motivating case of the Kac--Wakimoto characters, we require the Jacobi $\vartheta$ function given by 
\begin{equation}\label{Jacobit}\vartheta(z)=\vartheta(z;\tau):=-i\zeta^{-\frac12}q^{\frac18}(q)_{\infty}(\zeta )_{\infty}\left(\zeta^{-1}q\right)_{\infty}.\end{equation}
It is well-known that $\vartheta(z;\tau)$ is a holomorphic Jacobi form with multiplier of weight $\frac12$ and index $\frac12$. In particular, it satisfies $(2.1)$ with $\varepsilon=1$ and $m=\frac12$, namely, 
\begin{equation}
\label{JacobiThetaJacobiForm}
\vartheta(z+\lambda\tau+\mu)=(-1)^{\lambda+\mu}e^{-\pi i\left(\lambda^2\tau+2\lambda z\right)}\vartheta(z)
\end{equation}
for all $\lambda,\mu\in\Z$.

One of the most useful properties of holomorphic Jacobi forms, which in particular makes the study of their Fourier coefficients in $\zeta$ easy, is their theta decomposition. Namely, directly from the holomorphicity of $\phi$ and the elliptic transformation property, one finds that any holomorphic Jacobi form $\phi$ of index $m\in\N$ with $\varepsilon=0$ decomposes as 
\[\phi(z;\tau)=\sum_{\ell\smod{2m}}h_{\ell}(\tau)\vartheta_{m,\ell}(z;\tau),\]
where 
\[\vartheta_{m,\ell}(z;\tau):=\sum_{\substack{n\in\Z\\ n\equiv\ell\smod{2m}}}q^{\frac{n^2}{4m}}\zeta^n,\]
and 
\[h_{\ell}(\tau):=q^{-\frac{\ell^2}{4m}}\int_0^1\phi(z;\tau)e^{-2\pi i \ell z}dz.\]
Moreover, the modularity of $\phi$ implies that the coefficients $h_{\ell}$ ($0\leq\ell\leq 2m-1$) comprise a vector-valued modular form of weight $k-1/2$, where $k$ is the weight of the Jacobi form (see Chapter 2, Section 5 of \cite{EichlerZagier}).

When $\phi$ has poles, the situation is more complicated. Firstly, the Fourier coefficients depend on the imaginary part of $z$ and on the choice of path of integration.
To this end, following \cite{DMZ}, we define for $z_0\in\C$ and $\phi$ a function satisfying the transformation in \eqref{elliptictrans} with $m\in\frac12\Z$ and $\varepsilon\in\{0,1\}$, the (slightly modified) Fourier coefficients by 
\begin{equation}\label{FourierCoeffDefn} h_{\ell}(\tau) =h_{\ell,z_0}(\tau):=q^{-\frac{\ell^2}{4m}} \int_{z_0}^{z_0+1} \phi(z;\tau)\ko e^{-2\pi i\ell z} dz,\end{equation}
where $\ell\in m+\Z$.
Here, the path of integration is the straight line connecting $z_0$ and $z_0+1$ if there are no poles on this line. If there is a pole on the line which is not an endpoint, then we define the path to be the average of the paths deformed to pass above and below the pole. Finally, if there is a pole at an endpoint, note that the integral \eqref{FourierCoeffDefn} only depends on the imaginary part of $z_0$. Then we replace the path $[z_0,z_0+1]$ with $[z_0-\delta,z_0+1-\delta]$ for small $\delta$ so that there is not a pole at an endpoint, and then define the integral as above when there is a pole in the interior of the line.

To facilitate our study of meromorphic Jacobi forms, we require some notation. Firstly, for $z_0\in\C$, we let $P_{z_0}:= z_0+[0,1)\tau+[0,1)$ and we denote by $s_{z_0,\tau}$ the complete set of poles of $\phi$ in $P_{z_0}$. Finally, we denote the Laurent coefficients of $\phi$ around $u$ by
\begin{equation}\label{LaurentCoeffDefn}\phi(z)=:\frac{D_{n,u}}{\left(2\pi i(z-u)\right)^n}+\ldots+\frac{D_{1,u}}{2\pi i(z-u)}+O(1).\end{equation}
We note in passing that the Laurent coefficients are well--known to be quasimodular in the suppressed variable $\tau$ if $\phi$ is a Jacobi form. Roughly speaking, a quasimodular form is simply the constant term in $1/v$ of an almost holomorphic modular form, where an almost holomorphic modular form is a function of $\tau\in\mathbb H$ which transforms as a modular form and which is a polynomial in $1/v$ with holomorphic coefficients, where $\tau=u+iv$ with $u,v\in\mathbb R$.

\subsection{Partial theta functions and quantum modular forms}\label{QuantumS}

In this section, we recall some basic facts concerning quantum modular forms. We begin with the following definition, where $|_k$ is the usual Petersson slash operator (see \cite{ZagierQuantum} for more background on quantum modular forms).

\begin{definition}
For any cofinite set $\mathcal{Q}\subseteq\mathbb Q$, we say a function $f\colon \mathcal{Q}\rightarrow\C$ is a {\it quantum modular form} of weight $k\in\frac12\Z$ on a congruence subgroup $\Gamma$  if for all $\gamma\in\Gamma$, the cocycle
\[r_{\gamma}(\tau):=f|_k(1-\gamma)(\tau)\]

\noindent extends to an open subset of $\mathbb R$ and is analytically ``nice''. Here ``nice'' could mean continuous, smooth, real-analytic etc. 
\end{definition}

One of the most striking examples of a quantum modular form is Kontsevich's function $F(q)$, as studied by Zagier in \cite{ZagierVassiliev}, which is given by 

\begin{equation}\label{kontstrange}F(q):=\sum_{n\geq0}(q)_n.\end{equation}
 This function does not converge on any open subset of $\C$, but is a finite sum for $q$ any root of unity. Zagier's study of $F$ depends on the ``sum of tails'' identity 
\begin{equation}
\label{sumoftails}
\displaystyle\sum_{n\geq0}\left(\eta(\tau)-q^{\frac1{24}}\left(q\right)_n\right)=\eta(\tau)D\left(\tau\right)-\frac12\sum_{n\geq1}n\chi_{12}(n)q^{\frac{n^2-1}{24}},
\end{equation}
\noindent
where $\eta(\tau):=q^{1/24}(q)_{\infty}$,
$D(\tau):=-\frac12+\sum_{n\geq1}\frac{q^n}{1-q^n},$ and $\chi_{12}(\cdot):=\of{\frac{12}{\cdot}}$.
The key observation of Zagier is that in (\ref{sumoftails}), the values $\eta(\tau)$ and $\eta(\tau)D(\tau)$ vanish to infinite order as $\tau\rightarrow h/k$, so at a root of unity $\xi$, $F(\xi)$ is essentially the limiting value of the partial theta function $\sum_{n\geq1}n\chi_{12}(n)q^{\frac{n^2-1}{24}}$, which he showed has quantum modular properties \cite{ZagierVassiliev}.

In the decomposition of Jacobi forms of negative index, we encounter the more general partial theta functions $\vartheta^+_{\ell,M,\varepsilon}(z;\tau)$ defined in \eqref{PartialThetaDefn}.  
These functions turn out to yield quantum modular forms.

\begin{theorem}
\label{QuantumPartial}
For any $m\in-\frac12\N$, $\ell\in m+\Z$, $\varepsilon\in\{0,1\}$,  and $z\in\mathbb Q\tau+\mathbb Q$, the partial theta function $\theta^+_{\ell,\varepsilon,-m}(z;\tau)$ is (up to multiplication by a rational power of $q$) a quantum modular form of weight $1/2$ whose cocycles are real-analytic except at one point.

\end{theorem}
\begin{proof}
Note that if $z=\lambda\tau+\mu$ with $\lambda,\mu\in\mathbb Q$, then, up to multiplication by a constant and a multiple of $q$, and after rescaling $\tau\mapsto \tau/m$, we are led to study a partial theta series of the form
\begin{equation}\label{PartialThetaProofSketch}\sum_{n\geq0}(-1)^{n\varepsilon}q^{\left(n+\frac ab\right)^2}\end{equation}
for some $ a/b \in\mathbb Q$. As explained in \cite{BringmannRolen}, \eqref{PartialThetaProofSketch} is a ``holomorphic Eichler integral'' of the theta series 
\begin{equation}\label{ThetaProofSketch}\sum_{n\in\Z}(-1)^{n\varepsilon}\left(n+\frac ab\right)q^{\left(n+\frac ab\right)^2}.\end{equation}
As \eqref{ThetaProofSketch} is a cusp form of weight $3/2$, Theorem \ref{QuantumPartial} follows directly from Theorem 1.1 of \cite{BringmannRolen} (for a different perspective on these quantum modular forms, the reader is referred also to \cite{ChoiLimRhoades}).
\end{proof}

\section{Proofs of the main results}\label{ProofS}

We begin by giving the key properties of $F_{M,\varepsilon}$ needed for the proof of Theorem \ref{mainthm1}, both of which follow from direct calculations.

\begin{lemma}\label{lem}
Let $M\in\frac12\N$ and $\tau\in\Ha$. As a function of $u$, we have the elliptic transformation property
$$F_{M,\varepsilon}(z,u+\lambda\tau+\mu)= (-1)^{2M\mu+\lambda\varepsilon}e^{-2\pi iM(\lambda^2\tau+2\lambda u)}\ko F_{M,\varepsilon}(z,u),$$
for all $\lambda,\mu\in\Z$. Furthermore, as a function of  $u$, $F_{M,\varepsilon}(z,u)$ is a meromorphic function having only simple poles in $z+\Z\tau+\Z$ and residue $\frac1{2\pi i}$ in $u=z$. 
\end{lemma}

We are now in a position to prove our main result, Theorem \ref{mainthm1}.

\begin{proof}[Proof of Theorem \ref{mainthm1}]
Let $z\in\C$ be such that $\phi$ is holomorphic in $z$. Further let $z_0\in\C$ be such that $z\in P_{z_0}$ and that $\phi$ has no poles on the boundary of $P_{z_0}$. We consider the integral
$$ \int_{\partial P_{z_0}} \phi(v)\ko F_{-m,\varepsilon}(z,v)\ko dv,$$
which we now compute in two different ways: on the one hand, we find from equation \eqref{elliptictrans} and Lemma \ref{lem} that the integrand is both one-- and $\tau$--periodic. Hence we immediately see that the integral vanishes. On the other hand, we can use the Residue Theorem to give another evaluation of the integral. For this, we note that the poles of $v\mapsto \phi(v)\ko F_{-m,\varepsilon}(z,v)$ in $P_{z_0}$ are the poles of $\phi$ in $P_{z_0}$ together with $z$. In $v=z$ the function has residue $\frac 1{2\pi i}\ko \phi(z)$, so we have
$$ \int_{\partial P_{z_0}} \phi(v)\ko F_{-m,\varepsilon}(z,v)\ko dv = \phi(z) + 2\pi i\sum_{u\in s_{z_0,\tau}} \res_{v=u}\ko \bigl[ \phi(v)\ko F_{-m,\varepsilon}(z,v)\bigr]$$
and thus we get
$$\phi(z) =- 2\pi i\sum_{u\in s_{z_0,\tau}} \res_{v=u}\ko \bigl[ \phi(v)\ko F_{-m,\varepsilon}(z,v)\bigr].$$
Since the function $v\mapsto\phi(v)\ko F_{-m,\varepsilon}(z,v)$ is invariant under translation by a lattice point, so is $u\mapsto\res_{v=u}\ko \bigl[ \phi(v)\ko F_{-m,\varepsilon}(z,v)\bigr]$. Hence we can drop the condition that $z_0$ is such that $z\in P_{z_0}$ and take $z_0$ to be arbitrary (as long as there are no poles of $\phi$ on the boundary of $P_{z_0}$). The theorem now follows immediately by inserting the definition of the Laurent coefficients $D_{n,u}$.
\end{proof}

Before giving the proof of Theorem \ref{mainthm2}, we require the following properties of the partial theta functions under consideration, which follow from a direct calculation. 

\begin{lemma}
For all $\lambda,\mu\in\Z$, $\ell\in\frac12\Z$, and $M\in\frac12\N$, we have
\begin{equation}\label{PartThetaEllTrans} (-1)^{2\ell\mu}q^{M\lambda^2}\zeta^{2M\lambda}\ko \theta^+_{\ell,\varepsilon,M}\of{z+\lambda\tau+\mu} = \theta^+_{\ell-2M\lambda,\varepsilon,M}(z).
\end{equation}
 Furthermore, 
\begin{equation}\label{PartThetaTrans2} \theta^+_{\ell,\varepsilon,M}(z) - (-1)^{\varepsilon}q^{M}\zeta^{2M}\ko \theta^+_{\ell,\varepsilon,M}(z+\tau) = q^{\frac{\ell^2}{4M}}\zeta^{-\ell}.\end{equation}
\end{lemma}

\begin{proof}[Proof of Theorem \ref{mainthm2}] 
By the Residue Theorem we have
$$ \int_{\partial P_{z_0}} \phi(v)\ko \theta^+_{\ell,\varepsilon,-m}(v)\ko dv=2\pi i \sum_{u\in s_{z_0,\tau}} \res_{v=u}\ko \bigl[ \phi(v)\ko \theta^+_{\ell,\varepsilon,-m} (v)\bigr].$$

On the other hand, we can compute the integral directly. Since $\phi\, \theta^+_{\ell,\varepsilon,-m}$ is one--periodic (using the fact that $\ell\in m+\Z$), we find, using \eqref{elliptictrans} and \eqref{PartThetaTrans2} that
\begin{equation*}
\begin{split}
\int_{\partial P_{z_0}} \phi(v)\ko \theta^+_{\ell,\varepsilon,-m}(v)\ko dv&= \int_{z_0}^{z_0+1} \phi(v)\ko \theta^+_{\ell,\varepsilon,-m}(v)\ko dv - \int_{z_0+\tau}^{z_0+\tau+1} \phi(v)\ko \theta^+_{\ell,\varepsilon,-m}(v)\ko dv\\
&= \int_{z_0}^{z_0+1} \bigl( \phi(v)\ko \theta^+_{\ell,\varepsilon,-m}(v)- \phi(v+\tau)\ko \theta^+_{\ell,\varepsilon,-m}(v+\tau)\bigr)\ko dv\\
&= \int_{z_0}^{z_0+1} \hspace*{-0.1in}\phi(v) \bigl( \theta^+_{\ell,\varepsilon,-m}(v) - (-1)^{\varepsilon}e^{-2\pi im\of{\tau+2v}}\ko \theta^+_{\ell,\varepsilon,-m} (v+\tau)\bigr)\ko dv\\
&=  e^{-\frac{\pi i \ell^2 \tau}{2m}}\int_{z_0}^{z_0+1} \phi(v)\ko e^{-2\pi i\ell v}dv =h_{\ell,z_0}(\tau).
\end{split}
\end{equation*}
Comparing the two evaluations of the integral implies that
\[h_{\ell,z_0}(\tau) = 2\pi i \sum_{u\in s_{z_0,\tau}} \res_{v=u}\ko \bigl[ \phi(v)\ko \theta^+_{\ell,\varepsilon,-m} (v)\bigr].\]
The result then follows directly by inserting the definition of the Laurent coefficients into the last formula.
\end{proof}

\begin{proof}[Proof of Corollaries \ref{OldThm13Cor} and \ref{OldThm14Cor}] 
By \eqref{JacobiThetaJacobiForm}, we find that $\phi_{M,N}$ transforms according to \eqref{elliptictrans} with $\varepsilon=\varepsilon(N)$ and $m=\frac{M-N}2$. 
Further note that $\phi_{M,N}$ is a function whose only poles are poles of order $N$ in $\Z+\Z\tau$. Corollary \ref{OldThm13Cor} then follows directly by applying (1.2) with $z_0=-\frac12-\frac\tau2$. Similarly,  Corollary \ref{OldThm14Cor} follows directly by plugging into Theorem \ref{mainthm2}. 
\end{proof}

\end{document}